\documentclass{amsart}
\usepackage{amsthm, amssymb, amsmath, amsrefs}

\numberwithin{equation}{section}
\theoremstyle{theorem}
\newtheorem{theorem}{Theorem}[section]
\newtheorem{corollary}[theorem]{Corollary}
\newtheorem{lemma}[theorem]{Lemma}
\newtheorem{prop}[theorem]{Proposition}

\theoremstyle{definition}
\newtheorem{definition}[theorem]{Definition}
\newtheorem{remark}[theorem]{Remark}

\newcommand{\D}{\mathbb{D}}
\newcommand{\C}{\mathbb{C}}
\newcommand{\T}{\mathbb{T}}

\newcommand{\R}{\mathbb{R}}

\title{Testing von Neumann inequalities with nilpotent matrices}
\author{Greg Knese}
\address{Department of Mathematics, Washington University in St Louis, St Louis, MO 63130, USA}
\email{geknese@wustl.edu}

\thanks{Partially supported by NSF grant DMS-2247702}



\date{\today}

\begin{document}

\maketitle

\begin{abstract}
We give an elementary proof that
the Agler norm of a function is determined by its norm on
commuting tuples of nilpotent matrices.  The proof is a
variation on a standard cone separation argument.  The topic is closely
related to the Eschmeier-Patton-Putinar formulation
of Carath\'eodory-Fej\'er interpolation.
\end{abstract}

\section{Introduction}
In 1951, von Neumann proved that for a polynomial $p\in \C[z]$
and any contractive operator $T$ ($\|T\|\leq 1$) on a Hilbert space, we have
\[
\|p(T)\| \leq \sup_{z\in \D} |p(z)|
\]
where $\D$ is the unit disk in $\C$ \cite{vN}.
And\^{o}'s dilation theorem established the analogous 
inequality in two variables \cite{Ando}.  Namely, for a pair of commuting 
contractive operators $T=(T_1,T_2)$ and $p \in \C[z_1,z_2]$ 
we have
\[
\| p(T_1,T_2)\| \leq \sup_{|z_1|,|z_2|\leq 1} |p(z_1,z_2)|.
\]
Varopoulos first established that the inequality does not
generalize to three or more variables \cite{varo}; the appendix
of \cite{varo} contained an explicit counterexample due
to Kaijser and Varopoulos. 
On the other hand, knowing that a function satisfies a
von Neumann inequality implies a variety of useful properties
(such as a contractive transfer function realization formula); 
see \cite{Agler}, \cite{AMpick}, \cite{AMbook}, \cite{BTpick}.
Thus, it is of interest to compute the \emph{Agler norm}
of $p\in \C[z_1,\dots, z_d]$,
\[
\|p\|_A = \sup_{T} \|p(T)\|
\]
where the supremum is taken over all $d$-tuples
$T =(T_1,\dots, T_d)$ of commuting contractive operators
on a Hilbert space.

It would be desirable to compute the above norm over a smaller set
of operators and preferably over \emph{matrices} (our term of art for operators
operating on finite dimensional Hilbert spaces).
It is known that one can compute $\|p\|_A$ by taking
the supremum over $d$-tuples of simultaneously diagonalizable 
commuting contractive matrices with joint eigenspaces each of dimension $1$.
Something more general is proven in \cite{polyhedraAMY}, Theorem 6.1---we discuss
this further in Section \ref{sec:comments}.
It is also worth pointing out that many of the counterexamples
to the von Neumann inequality in 3 or more variables use
nilpotent matrices (and sometimes nilpotent partial isometries).
This is the case in the original paper of Varopoulos \cite{varo} as
well as Crabb-Davie \cite{crabb}, Dixon \cite{dixon}, and Holbrook's optimal (as far as dimension)
example \cite{holbrook}.

Recently, B.\ Cole proved that von Neumann's inequality can be tested
using $d$-tuples of commuting contractive nilpotent matrices. 
The proof employed ideas from $C^*$-algebras and
harmonic analysis.  
M. Hartz has also privately indicated another proof of this.
The next result can thus be considered part of the folklore---so not due to
the present author---but also not published anywhere.  
It should also be pointed out that related concepts have been 
studied in the setting of nc *-polynomials/continuous functions of single operators;
see Courtney-Sherman \cite{Sherman}.  

\begin{theorem} \label{colethm}
Suppose $p \in \C[z_1,\dots,z_d]$.  Then, 
\[
\|p(T)\| \leq 1
\]
for every $d$-tuple of commuting contractive operators
$T=(T_1,\dots, T_d)$ on a Hilbert space if and only
if for every $d$-tuple of commuting contractive nilpotent
matrices $M=(M_1,\dots, M_d)$ we have
\[
\|p(M)\| \leq 1.
\]
\end{theorem}

In this article we give a different proof of this theorem.
Our proof involves combining the $d$-variable 
Carath\'eodory-Fej\'er coefficient interpolation theorem
of Eschmeier-Patton-Putinar \cite{EPP}
with a cone separation argument.
Part of the argument is that we can simplify things
further and restrict our attention to a special
class of commuting nilpotent matrices.

\begin{definition} \label{simplenil}
Let $T= (T_1,\dots, T_d)$ be a $d$-tuple of commuting $n\times n$ matrices.
We say $T$ is a \emph{simple $N$-nilpotent $d$-tuple} if $\C^n$
has a spanning set indexed in the following way
\begin{equation} \label{spanset}
\{\vec{b}_{\alpha} : \alpha=(\alpha_1,\dots, \alpha_d) \in \mathbb{N}_0^d, 
|\alpha| = \sum_{j=1}^{d} \alpha_j \leq N \} \subset \C^n
\end{equation}
such that
\begin{equation} \label{Tdef}
T_j \vec{b}_{\alpha} = \begin{cases} \vec{b}_{\alpha+e_j} & \text{ if } |\alpha| \leq N-1 \\
0 & \text{ if } |\alpha| = N. \end{cases}
\end{equation}
where $e_j$ is the vector with all zeros except a $1$ in the $j$-th position. $\diamond$
\end{definition}

Necessarily we will have for $\beta \in \mathbb{N}_0^d$
\[
T^{\beta} \vec{b}_{\alpha} = \begin{cases} \vec{b}_{\alpha+\beta} & \text{ if } |\alpha+\beta| \leq N \\
0 & \text{ if } |\alpha+\beta| > N. \end{cases}
\]
This definition gives an abstract formulation of a certain type of commuting
$d$-tuple of nilpotent
operators.  The recipe above yields a
well-defined $d$-tuple of commuting
nilpotents if the $\vec{b}_{\alpha}$ vectors form
a basis (and not just a spanning set), but 
it will be convenient to allow 
degeneracies.  To build all of the simple $N$-nilpotent
$d$-tuples would be non-trivial because in the case where
the $\vec{b}_{\alpha}$'s are not linearly independent one would
need to check that \eqref{Tdef} is compatible with the relations
between the spanning vectors.

Also, the above setup is equivalent to having an inner product on 
a quotient $\C[z_1,\dots, z_d]/I$ where $I$ is an ideal
containing the ideal $(z)^{N+1}$ of polynomials vanishing
to order at least $N+1$.  In this formulation, $T_j$ is multiplication
by $z_j$ and in what follows we will be interested in the situation
where the operators $T_1,\dots, T_d$ are contractive.

We will establish Theorem \ref{colethm} via the following series of
equivalences.  The main novelty in this paper is $(1) \implies (2)$.

\begin{theorem} \label{nilthm}
Let $p \in \C[z_1,\dots, z_d]$ have total degree at most $N$.
Write $p(z) = \sum_{|\alpha| \leq N} p_{\alpha} z^{\alpha}$.
The following are equivalent:
\begin{enumerate}
\item 
For every simple $N$-nilpotent $d$-tuple of contractions $T$
we have $\|p(T)\|\leq 1$.
\item 
There exist positive semi-definite matrices $A^1,\dots, A^d$
such that
\[
\delta_{0,\alpha,\beta} - \overline{p_{\alpha}} p_{\beta}
=
\sum_{j=1}^{d} (A^j_{\alpha,\beta} - A^j_{\alpha-e_j, \beta-e_j})
\]
where $A^j_{\alpha,\beta} = 0$ if an index is out of bounds
and $\delta_{0,\alpha,\beta} = 1$ exactly when $\alpha=\beta=0$ and
$0$ otherwise.

\item 
There exist vector
polynomials $A^1(z),\dots, A^d(z)$ of degree at most $N$ such that
\[
1-|p(z)|^2 = \sum_{j=1}^{d} (1-|z_j|^2) |A^j(z)|^2 \text{ mod } ((z)^{N+1},(\bar{z})^{N+1}).
\]
\item 
There exists a rational inner function 
$\phi:\D^d\to \D$ such that 
\[
p(z) - \phi(z) \text{ vanishes to order at least } N+1
\]
and such that there exist positive semi-definite kernels $K_1,\dots, K_d:\D^d\times \D^d \to \C$
so that
\begin{equation} \label{agdecomp}
1-\overline{\phi(w)}\phi(z) = \sum_{j=1}^{d} (1-\bar{w}_j z_j)K_j(z,w).
\end{equation}
\end{enumerate}
\end{theorem}

\begin{remark}
We have several remarks.
\begin{itemize}
\item Here $(z)^{N+1}$ denotes the ideal in $\C[z_1,\dots, z_d]$ generated
by $\{ z^{\alpha}: |\alpha| = N+1\}$
and $((z)^{N+1}, (\bar{z})^{N+1})$ denotes the ideal in $\C[z_1,\dots, z_d, \bar{z}_1,\dots, \bar{z}_d]$
generated by $\{z^{\alpha}, \bar{z}^{\alpha}: |\alpha|=N+1\}$.
\item Note that the theorem characterizes when a polynomial $p$
is the truncation of the power series of an analytic function $f$ with Agler norm at most
one (it does \emph{not} characterize when $p$ itself has Agler
norm at most $1$).  
\item The theorem offers a conceptual improvement of 
the Carath\'eodory-Fej\'er theorem of 
Eschmeier-Patton-Putinar in the sense that the von Neumann
inequality (item 1) is the easiest condition to disprove in 
the situation where all of the equivalent conditions are false.
Specifically, to prove $p(z)$ is not the truncation (at degree $N+1$)
of the power series of an analytic function with Agler
norm at most $1$, one needs only exhibit a simple $N$-nilpotent contractive
$d$-tuple $T$ such that $\|p(T)\| >1$.
\item The paper of Ball-Li-Timotin-Trent \cite{BLTT} also contains
a formulation of the Carath\'eodory-Fej\'er interpolation theorem on 
polydisks. See Theorem 5.1 of \cite{BLTT}.
\item On the other hand, the known equivalent conditions (2),(3) are good
for demonstrating that the equivalent conditions are all true.
It is also worth pointing out that checking item 1 is a matter
of \emph{testing} a finite dimensional family of inequalities
while checking item 2 is a \emph{search} over a finite dimensional family
of positive matrices. 

\item Regarding item (4) in Theorem \ref{nilthm}, 
a \emph{rational inner function} is a rational function $\phi(z) = Q(z)/R(z)$ with
$R(z) \ne 0$ for $z\in \D^d$ and $|R(z)|= |Q(z)|$ on $\T^d$.
The formula \eqref{agdecomp} is called an Agler decomposition
and it automatically implies that $\|\phi(T)\| \leq 1$ for any
 $d$-tuple of commuting strictly contractive operators $T$.
 This follows from the hereditary functional calculus 
 (see for instance Section 8 of \cite{CW}).
The ``rational inner'' aspect in item (4) does not get used here but it is worth including 
nonetheless.  See \cite{rifsurvey} for more information about
rational inner functions.

\item 
We also should point out that there is a very strong improvement 
to this theorem in one dimension and it is the original Carath\'eodory-Fej\'er
interpolation theorem. 
For $d=1$,  we can replace condition (1) in Theorem \ref{nilthm}
with the condition that
\[
\|p(T)\| \leq 1
\]
for the single $(N+1)\times (N+1)$ matrix
defined as an operator by
\[
T \vec{e}_j = \vec{e}_{j+1} \text{ for } j=0,\dots, N-1
\]
\[
T \vec{e}_{N} = 0
\]
for an orthonormal basis $\vec{e}_0,\dots, \vec{e}_{N}$.
Indeed, $\|p(T)\| \leq 1$
is equivalent to $\|p(T)^*\| \leq 1$, which we can write as
\[
|\sum_{j} \bar{p}_j (T^j)^* \sum_{k} \bar{a}_k \vec{e}_k|^2 \leq \sum_{j=0}^{N} |a_j|^2
\]
for all scalars $a_0,\dots, a_N$.  
(Certain conjugations will help simplify later expressions.)
This is equivalent to
\[
\begin{aligned}
|\sum_{j,k} p_j a_k \vec{e}_{k-j}|^2 
&=
|\sum_{j,k} p_j a_{k+j} \vec{e}_k|^2\\
&=
\sum_{k} |\sum_{j} p_j a_{k+j}|^2\\
&=
\sum_{k} |\sum_{j} p_{j-k} a_j|^2\\
&=
\sum_k \sum_{j,i} \bar{p}_{j-k} p_{i-k} \bar{a}_j a_i\\
&=
\sum_{j,i} \left(\sum_k \bar{p}_{j-k} p_{i-k}\right) \bar{a}_j a_i \\
&\leq 
\sum_{j,i} \delta_{j,i} \bar{a}_j a_i
\end{aligned}
\]
Again we are using the convention that whenever an index is out of bounds, then
that quantity is treated as zero (i.e. $\vec{e}_j$ for $j$ negative).
The last inequality is equivalent to
\[
\left(\sum_{k} \bar{p}_{j-k} p_{i-k}\right)_{j,i} \leq I
\]
with ``$\leq$'' in the sense of  the partial order on self-adjoint matrices.
If we define $A_{j,i} = \delta_{j,i} - \sum_{k} \bar{p}_{j-k} p_{i-k}$
then $(A_{j,i})_{j,i}$ is positive semi-definite and
\[
A_{j,i} - A_{j-1,i-1}  = \delta_{0,j,k} - \bar{p}_j p_i.
\]
This is condition (2) in Theorem \ref{nilthm}
in the case $d=1$.
Note that the condition $\|p(T)\|\leq 1$
is equivalent to the following Toeplitz matrix being contractive:
\[
\begin{pmatrix} 
p_0 & p_1 & \dots & p_N \\
0 & p_0 & \ddots & \vdots \\
\vdots & \vdots &  \ddots & p_1 \\
0 & 0 & \cdots & p_0\end{pmatrix}.
\]
This is a common formulation of the
classical Carath\'eodory-Fej\'er theorem. $\diamond$
\end{itemize}

\end{remark}

Theorem \ref{colethm}, follows from 
Theorem \ref{nilthm}.  
In fact, we can prove the following more general fact.

\begin{theorem} \label{genthm}
Let $f\in \D^d \to \C$ be analytic.  We have
\[
\|f(T)\| \leq 1
\]
for all $d$-tuples $T$ of commuting strictly contractive operators on a Hilbert space
if and only if the same inequality holds for all $d$-tuples of commuting 
strictly contractive nilpotent matrices.
\end{theorem}

Theorem \ref{colethm} does not require \emph{strict} contractions 
but by continuity we can remove this requirement when our functions
are continuous up to the boundary, as with polynomials.
Of course, in Theorem \ref{colethm} since the statement is about
polynomials, one might expect that it is not necessary to 
check $\|p(M)\|\leq 1$ for nilpotents of \emph{all} sizes
since $p$ has only finitely many coefficients.
This is an interesting question but we do not address it here;
our goal is to present and prove a precise structural relationship
between $d$-tuples of commuting nilpotents 
and Taylor polynomial interpolation (or rather Carath\'eodory-Fej\'er interpolation)
of analytic functions with Agler norm at most $1$.

\begin{remark} \label{hartzex}
After initial posting of this article, M. Hartz pointed out that
all nilpotent matrix sizes are necessary to determine the Agler norm
for the simple example one variable polynomial $p(z) = (1+z)/2$.
\end{remark}

\section{Proofs}

To begin, let us see why Theorem \ref{genthm} follows from Theorem \ref{nilthm}.

\begin{proof}[Proof of Theorem \ref{genthm} from Theorem \ref{nilthm}]
Suppose $f:\D^d \to \C$ is analytic and satisfies
\[
\|f(M)\| \leq 1
\]
for every $d$-tuple of commuting strictly contractive nilpotent
matrices $M$.
Then the degree $N$ Taylor polynomial $f_N$ of $f$ satisfies the hypotheses of
Theorem \ref{nilthm}.  So, there exists a rational inner function $\phi_N$
with an Agler decomposition as in \eqref{agdecomp}
and such that $f_N-\phi_N$ vanishes to order $N+1$.
Then, $\phi_N - f$ vanishes to order $N+1$.
This implies $\phi_N$ converges locally uniformly
to $f(z)$.  A standard normal families argument proves
that a local uniform limit $f$ of functions (in this case $\phi_N$)
with Agler decompositions will itself possess an 
Agler decomposition.  As mentioned this implies
$f(T)$ is contractive for any $d$-tuple $T$ of
commuting strictly contractive operators.
\end{proof}

Next, we prove Theorem \ref{nilthm} using Lemma \ref{Cclosed} to come later.

\begin{proof}[Proof of Theorem \ref{nilthm}]
The main novelty is the proof that
(1) implies (2).
The remaining implications are essentially known or
straightforward.
To be precise, (2) and (3)
are equivalent simply through extracting coefficients
of $z^{\alpha} \bar{z}^{\beta}$.
The implication (3) $\implies$ (4)
is a lurking isometry argument found in Eschmeier-Patton-Putinar \cite{EPP}.
The implication (4) $\implies$ (1)
follows from first observing that \eqref{agdecomp}
implies $\phi(T)$ is contractive
for any $d$-tuple $T$ of commuting strictly contractive
operators.
Then, $\phi(T) = p(T)$ for a simple $N$-nilpotent 
$d$-tuple $T$ of strict contractions since $T^{\alpha} = 0$ for $|\alpha|>N$, so that $\|p(T)\|\leq 1$.
One can then take a limit
to non-strict contractions.
We proceed to prove (1) implies (2).

Suppose condition (1) holds.  
Let
\begin{equation} \label{Ndef}
[N] = \{\alpha \in \mathbb{N}_0^d: |\alpha|\leq N\} \text{ and } n := \#[N] = \binom{d+N}{N}.
\end{equation}
Consider the following cone of self-adjoint matrices whose rows and
columns are indexed by $[N]$
\[
\mathcal{C} = 
\left\{\left(\sum_{j=1}^{d} A^{j}_{\alpha,\beta} - A^j_{\alpha-e_j, \beta-e_j}\right)_{\alpha, \beta\in [N]}:
A^1,\dots, A^d \text{ are positive semi-definite}\right\}.
\]
Again $A^j_{\alpha,\beta}$ is treated as zero whenever an
index is out of the valid range $\alpha,\beta \in [N]$. 
It suffices to prove that 
$X = (\delta_{0,\alpha,\beta} - \bar{p}_{\alpha} p_{\beta})_{\alpha,\beta\in [N]}$
belongs to $\mathcal{C}$.

By Lemma \ref{Cclosed} below, $\mathcal{C}$ is closed.
Now, if $X \notin \mathcal{C}$, 
then there exists a separating hyperplane; namely there
exists a self-adjoint matrix $B = (B_{\alpha,\beta})_{\alpha,\beta\in [N]}$
such that
for all $C \in \mathcal{C}$, $\text{tr}(CB^t) \geq 0$
and $\text{tr}(XB^t) <0$.

First, we show $B$ is positive semi-definite.
Let $a \in \C^n$ be a column vector and set $A^1 = \bar{a} a^t = (\bar{a}_{\alpha} a_{\beta})_{\alpha,\beta \in [N]}$
and $A^2=\cdots = A^d = 0$.
Then, for $C \in \mathcal{C}$ given by
\[
C_{\alpha, \beta} = A^1_{\alpha,\beta} - A^1_{\alpha-e_1,\beta-e_1} 
= \bar{a}_{\alpha} a_{\beta} - \bar{a}_{\alpha-e_1} a_{\beta-e_1}
\]
we have
\[
\begin{aligned}
0 &\leq \text{tr}(CB^t) = \sum_{\alpha,\beta \in [N]} (\bar{a}_{\alpha} a_{\beta} - \bar{a}_{\alpha-e_1,\beta-e_1}) B_{\alpha, \beta} \\
&= \sum_{\alpha,\beta} \bar{a}_\alpha a_{\beta} B_{\alpha,\beta} 
- \sum_{\alpha,\beta} \bar{a}_\alpha a_{\beta} B_{\alpha+e_1, \beta+e_1}. \\
\end{aligned}
\]
Applying the same reasoning after successively setting 
\[
A^1_{\alpha,\beta} = \bar{a}_{\alpha-e_1} a_{\beta-e_1}, 
\bar{a}_{\alpha-2e_1} a_{\beta-2e_1}, \dots
\] 
we see that for $k=0,1,2,\dots$ 
\[
0\leq \sum_{\alpha,\beta} \bar{a}_{\alpha} a_{\beta} B_{\alpha+ke_1,\beta+ke_1} 
- \sum_{\alpha,\beta} \bar{a}_\alpha a_{\beta} B_{\alpha+(k+1)e_1, \beta+(k+1)e_1}.
\]
For large enough $k$ all of these entries are zero.  
Summing over all $k$ shows that
\[
0\leq \sum_{\alpha,\beta} \bar{a}_\alpha a_{\beta} B_{\alpha,\beta}; 
\]
namely, $B$ is positive semi-definite.
Therefore, we can factor $B_{\alpha,\beta} = \vec{b}_{\alpha}^* \vec{b}_{\beta}$
for some family of vectors $\vec{b}_{\alpha} \in \C^r$ where $r$ is the rank of $B$.
By the argument above
\[
\sum_{\alpha,\beta} \bar{a}_\alpha a_{\beta} \vec{b}_{\alpha}^*\vec{b}_{\beta}
\geq 
\sum_{\alpha,\beta} \bar{a}_\alpha a_{\beta} \vec{b}_{\alpha+e_j}^*\vec{b}_{\beta+e_j}
\]
which can be written as
\[
\left| \sum_{\alpha} a_{\alpha} \vec{b}_{\alpha} \right|^2 \geq 
\left| \sum_{\alpha} a_{\alpha} \vec{b}_{\alpha+e_j} \right|^2
\]
Therefore the maps $T_j: \C^r \to \C^r$
\[
T_j \vec{b}_{\alpha} = \begin{cases} \vec{b}_{\alpha+e_j} & \text{ for } |\alpha+e_j| \leq N \\
0 & \text{ otherwise} \end{cases}
\]
extend linearly and in a well-defined way to 
form a simple $N$-nilpotent $d$-tuple of contractions.
By our assumption (item (1) in the theorem statement), 
\[
\|p(T)\| \leq 1
\]
which means that for all scalars $a_\alpha \in \C$
\[
\left| p(T)  \sum_{\alpha} a_{\alpha} \vec{b}_{\alpha} \right|^2
\leq
\left| \sum_{\alpha} a_{\alpha} \vec{b}_{\alpha} \right|^2.
\]
The left side equals
\[
\begin{aligned}
\left| \sum_{\gamma} p_{\gamma} T^{\gamma} \sum_{\alpha} a_{\alpha} \vec{b}_{\alpha}
\right|^2
&=
\left| \sum_{\gamma,\alpha} p_{\gamma} a_{\alpha} \vec{b}_{\alpha+\gamma}
\right|^2\\
&=
\left| \sum_{\gamma,\alpha} p_{\gamma-\alpha} a_{\alpha} \vec{b}_{\gamma}
\right|^2\\
&=
\sum_{\alpha,\beta} \bar{a}_\alpha a_{\beta} 
\sum_{\gamma,\delta} \bar{p}_{\gamma-\alpha} p_{\delta-\beta}
\vec{b}_{\gamma}^* \vec{b}_{\delta}.
\end{aligned}
\]
From this we see that
\[
\left( \vec{b}_{\alpha}^*\vec{b}_{\beta} - \sum_{\gamma,\delta} \bar{p}_{\gamma-\alpha} p_{\delta-\beta}
\vec{b}_{\gamma}^* \vec{b}_{\delta}\right)_{\alpha,\beta \in [N]}
\]
is positive semi-definite.
The $(0,0)$ entry is therefore non-negative:
\[
0\leq |\vec{b}_{0}|^2 - \sum_{\gamma,\delta} \bar{p}_{\gamma} p_{\delta} \vec{b}_{\gamma}^* \vec{b}_{\delta}
=|\vec{b}_{0}|^2 - \left| \sum_{\gamma} p_{\gamma} \vec{b}_{\gamma}\right|^2.
\]
But recall $X = (\delta_{0,\alpha,\beta} - \bar{p}_{\alpha} p_{\beta})_{\alpha,\beta}$
so that
\[
\text{tr}(XB^t) = 
\sum_{\alpha,\beta} (\delta_{0,\alpha,\beta} - \bar{p}_{\alpha} p_{\beta}) 
\vec{b}_{\alpha}^* \vec{b}_{\beta}
= |\vec{b}_0|^2 - \left| \sum_{\alpha} p_{\alpha} \vec{b}_{\alpha}\right|^2
\]
is both non-negative and negative---a contradiction.
Therefore, $X$ belongs to the cone $\mathcal{C}$.

\end{proof}

We used the following lemma in the previous proof.

\begin{lemma}\label{Cclosed}
Recall $[N]$ and $n$ defined in \eqref{Ndef}.
Consider the following cone of $n\times n$ 
self-adjoint matrices whose rows and
columns are indexed by the multi-indices $\alpha \in [N]$  
\[
\mathcal{C} = 
\left\{\left(\sum_{j=1}^{d} A^{j}_{\alpha,\beta} - A^j_{\alpha-e_j, \beta-e_j}\right)_{\alpha, \beta\in [N]}:
A^1,\dots, A^d \text{ are positive semi-definite}\right\}.
\]
Again, $A^j_{\alpha,\beta}$ is treated as zero whenever an
index is out of the valid range $\alpha,\beta \in [N]$.
The cone $\mathcal{C}$ is closed.
\end{lemma}
\begin{proof}
Suppose 
\[
C(k) := \left(\sum_{j=1}^{d} A^{j}_{\alpha,\beta}(k) - A^j_{\alpha-e_j, \beta-e_j}(k) \right)_{\alpha, \beta \in [N]}
\]
is a sequence of self-adjoint matrices in $\mathcal{C}$
defined in terms of positive semi-definite matrices 
$A^1(k),\dots, A^d(k)$ and suppose $C(k) \to C$ as $k\to\infty$
where $C$ is a self-adjoint matrix.
We must show $C \in \mathcal{C}$.
Since $(C(k))_{k=0}^{\infty}$ converges, we can bound it
\[
a_1 I \leq C(k) \leq a_2 I
\]
for $a_1,a_2 \in \R$.  
Performing a shift $\gamma \in \mathbb{Z}^d$ on a self-adjoint matrix
\[
(C_{\alpha,\beta})_{\alpha,\beta \in [N]}
\mapsto (C_{\alpha- \gamma,\beta - \gamma})_{\alpha,\beta\in [N]}
\]
does not change these bounds because they amount to restricting
our matrix to a block.
(Again, note that with the convention that out-of-bounds indices give
$0$, the matrix on the right will have blocks of zeros.)
If we perform all possible shifts with $\gamma\in [N]$ to $C(k)$ and
then sum
we get
\[
a_1n I \leq \sum_{\gamma \in [N]} 
\left(\sum_{j=1}^{d} A^{j}_{\alpha-\gamma,\beta-\gamma}(k) - 
A^j_{\alpha-e_j-\gamma, \beta-e_j-\gamma}(k) \right)_{\alpha,\beta \in [N]}
\leq a_2 n I
\]
and because of telescoping sums we get
\[
a_1n I \leq \left(\sum_{j=1}^{d} A^{j}_{\alpha,\beta}(k) \right)_{\alpha,\beta \in [N]}
\leq
a_2n I.
\]
Namely,
\[
a_1 n I \leq \sum_{j=1}^{d} A^j(k) \leq a_2 n I.
\]
Therefore, the positive semi-definite matrices $A^j(k)$, $j=1,\dots, d$,
are bounded independent of $k$ and we can select subsequences
that converge to positive semi-definite matrices $A^1,\dots, A^d$.
There is no harm in replacing the entire sequence with this subsequence.
Then, necessarily as $k\to \infty$
\[
C(k) \to C = \left(\sum_{j=1}^{d} A^{j}_{\alpha,\beta} - A^j_{\alpha-e_j, \beta-e_j} \right)_{\alpha, \beta} \in \mathcal{C}.
\]
This proves that $\mathcal{C}$ is closed.

\end{proof}

\section{Agler-Pick interpolation and simultaneously diagonalizable tuples} \label{sec:comments}

As mentioned earlier, it is also known that one can test von Neumann's
inequality using $d$-tuples of simultaneously diagonalizable 
commuting contractions with joint eigenspaces each of dimension $1$.
Something more general is proven in \cite{polyhedraAMY}.
A proof can be given along the lines above again 
using a cone separation argument.  The replacement for
Theorem \ref{nilthm} would be the following theorem which
is more or less known but not explicitly stated in this form.

\begin{theorem} \label{diagthm}
Let $S \subset \D^d$ be finite and let $f:S \to \C$ 
be a function.  Set $n= \# S$.
The following are equivalent.
\begin{enumerate}
\item For every $d$-tuple $T$
of commuting, contractive, simultaneously diagonalizable matrices $T$
whose joint eigenspaces each have
dimension $1$ and $\sigma(T) \subset S$
we have
\[
\|f(T)\| \leq 1.
\]
\item There exist positive semi-definite $n \times n$ matrices $A^1,\dots, A^d$
whose rows and columns we index by $S$ 
such that for $z,w \in S$
\begin{equation} \label{finAg}
1- \overline{f(w)} f(z) = \sum_{j=1}^{d} (1-\bar{w}_j z_j) A^j_{z,w}.
\end{equation}

\item There exists a rational inner function $\phi:\D^d \to \D$ such that 
\[
\phi(z_j) = f(z_j) \text{ for } j=1,\dots, n
\]
and there exist positive semi-definite kernels $K^1,\dots, K^d$ on $\D^d\times \D^d$
such that
\[
1-\overline{\phi(w)} \phi(z) = \sum_{r=1}^{d} (1-\bar{w}_j z_j) K^j(z,w).
\]
\end{enumerate}

\end{theorem}

Notice that even though $f$ is simply a function
on a finite set, the hypotheses on $T$ make it possible
to define $f(T)$.   The relevant cone consists of
expressions on the right side of \eqref{finAg}.
Our argument with telescoping sums is replaced with an
argument involving the Schur product theorem (i.e.\ the entrywise product
of positive semi-definite matrices is positive semi-definite).

While Theorem \ref{nilthm} was an improvement of the Eschmeier-Patton-Putinar
extension of the Carath\'eodory-Fej\'er theorem, the above theorem
is an improvement of Agler's Pick interpolation theorem.
Again, the conceptual improvement comes from the fact that item 1 is the easiest condition
to disprove (if false) and also shows that interpolation can be checked by testing
a finite dimensional family of inequalities. Item 2 is dual; it is the easiest condition
to prove if true (i.e. simply exhibit the appropriate positive matrices)
and it involves a search over a finite dimensional family.

The companion to Theorem \ref{genthm} would then be the following.
\begin{theorem}\label{genthmdiag}
Let $f:\D^d \to \C$ be analytic. We have
\[
\|f(T)\| \leq 1
\]
for all $d$-tuples $T$ of commuting strictly contractive operators on a Hilbert space
if and only if the same inequality holds for all $d$-tuples of commuting contractive
simultaneously diagonalizable matrices.
\end{theorem}

Combining Theorems \ref{genthm} and \ref{genthmdiag} we can state the following corollary.

\begin{corollary}\label{gencor}
Let $f: \D^d \to \C$ be analytic and $c>0$. 
The following are equivalent:
\begin{itemize}
\item $\|f\|_A > c$ 
\item There exists a $d$-tuple $T$ of commuting contractive simultaneously diagonalizable 
matrices such that $\|f(T)\| > c$.
\item 
There exists a $d$-tuple $M$ of commuting contractive nilpotent matrices such that $\|f(M)\| >c$.
\end{itemize}
\end{corollary}

\begin{remark}
The paper Lotto-Steger \cite{Lotto}
exhibits a polynomial $p$ (actually the usual Kaijser-Varopoulos polynomial \cite{varo})
 in three variables and a commuting contractive $3$-tuple of $5\times 5$ matrices $T$ that are simultaneously
 diagonalizable such that $\|p(T)\| > \sup_{\D^3} |p|$.
 The example is constructed by perturbing the nilpotent matrices from
 the original Kaijser-Varopoulos example.  
 This example (at the time) was designed to show that the von Neumann 
 inequality could fail with simultaneously diagonalizable
 matrices.  However, Corollary \ref{gencor} shows that an explicit
 perturbation is not necessary.  
 Indeed, Theorem \ref{genthm} and
 Theorem \ref{genthmdiag}
 prove that for a polynomial $p \in \C[z_1,\dots, z_d]$
 there exists a $d$-tuple of commuting contractive nilpotent matrices $M$
 such that $\|p(M)\| > \sup_{\D^d} |p|$ if and only if
 there exists a $d$-tuple of commuting contractive simultaneously
 diagonalizable matrices $T$ such that
 $\|p(T)\| > \sup_{\D^d} |p|$.
 However, the abstract nature of the approach taken here
 would make it difficult to explicitly go between $M$ and $T$.
 
 The paper of Holbrook-Omladič \cite{HO} is also relevant here---they 
 study when commuting $d$-tuples of matrices can
 be approximated by simultaneously diagonalizable tuples 
 with some motivation coming from von Neumann inequalities.
 A ``von Neumann inequality'' can be broadly construed as taking a family
 of functions $f$ and a family of $d$-tuples of commuting operators $T$
 and wanting to compute the supremum of $\|f(T)\|$.
  Approximation arguments are relevant if one wants to construct
  ``minimal'' or simple counterexamples to von Neumann inequalities (broadly construed)
  as far as the operators are concerned, but if we are simply interested in
 ``minimal'' counterexamples as far as the function is concerned then
 the operators can either be taken to be simultaneously diagonalizable or
 nilpotent matrices. 
  $\diamond$
\end{remark}

\section{Acknowledgments}
Thanks to Brian Cole for sharing his result in the analysis
seminar at Washington University in St. Louis in spring 2024 and thanks
as well for useful conversations.  Thanks to John M$^{\text{c}}$Carthy
for pointing out that \cite{polyhedraAMY} contains the fact
that the Agler norm can be computed with 
simultaneously diagonalizable matrices.
Thanks very much to Michael Hartz for comments on this paper especially
regarding his proof of Theorem \ref{colethm} and Remark \ref{hartzex}.
Thank you to David Sherman for pointing on the reference \cite{Sherman}.
Finally, thank you to the anonymous referee for a thoughtful report and for 
pointing our several typos.

\end{document}